\documentclass[11pt,a4paper,oneside,english,british]{amsart}

\usepackage[T1]{fontenc}
\usepackage[applemac]{inputenc}
\usepackage{babel}
\usepackage{textcomp}
\usepackage{amstext}
\usepackage{amsthm}
\usepackage{multicol}
\usepackage{amssymb}
\usepackage{stmaryrd}
\usepackage[unicode=true,
 bookmarks=false,
 breaklinks=false,pdfborder={0 0 1},backref=section,colorlinks=false]
 {hyperref}

\makeatletter

\pdfpageheight\paperheight
\pdfpagewidth\paperwidth

\newtheorem{thm}{Theorem}
\newtheorem{defn}{Definition}
\newtheorem{lem}{Lemma}
\newtheorem{pro}{Proposition}
\newtheorem{rk}{Remark}

\usepackage{pgf}
\usepackage{tikz}\usetikzlibrary{arrows}

\usepackage{mathrsfs}

\usepackage{babel}
\addto\captionsbritish{}
\addto\captionsbritish{}
\addto\captionsenglish{}
\addto\captionsenglish{}

\AtBeginDocument{
  
}

\makeatother

\addto\captionsbritish{}
\addto\captionsbritish{}
\addto\captionsbritish{}
\addto\captionsbritish{}
\addto\captionsbritish{}
\addto\captionsenglish{}
\addto\captionsenglish{}
\addto\captionsenglish{}
\addto\captionsenglish{}
\addto\captionsenglish{}
\addto\captionsenglish{}

\begin{document}
\title[An application of SEIR model to the COVID-19 spread]{An application of discrete-time SEIR model to the COVID-19 spread}
\author{u.a.rozikov, s. k. shoyimardonov}
\begin{abstract}
The Susceptible-Exposed-Infectious-Recovered (SEIR) model is applied in several countries to ascertain the spread of the coronavirus disease 2019 (COVID-19).
We consider discrete-time  SEIR  epidemic model  in a closed system which does not account for births or deaths, total population size under consideration is constant. This dynamical system generated by a non-linear evolution operator depending on four parameters. Under some conditions on parameters we reduce the evolution operator to a quadratic stochastic operator (QSO) which maps 3-dimensional simplex to itself.  We show that the QSO has uncountable set of fixed points (all laying on the boundary of the simplex). It is shown that all trajectories of the dynamical system (generated by the QSO) of the SEIR model are convergent (i.e. the QSO is regular). Moreover, we discuss the efficiency of the model for Uzbekistan.

\end{abstract}


\address{ U.Rozikov$^{a,b,c}$\begin{itemize}
 \item[$^a$] V.I.Romanovskiy Institute of Mathematics of Uzbek Academy of Sciences;
\item[$^b$] AKFA University, 1st Deadlock 10, Kukcha Darvoza, 100095, Tashkent, Uzbekistan;
\item[$^c$] Faculty of Mathematics, National University of Uzbekistan.
\end{itemize}}
\email{rozikovu@yandex.ru}

\address{Sobirjon Shoyimardonov. \ \ V.I.Romanovskiy Institute of Mathematics of Uzbek Academy of Sciences, Tashkent, Uzbekistan.}
\email{shoyimardonov@inbox.ru}

\maketitle
\selectlanguage{british}%

\section{Introduction}

Since December 2019, healthcare systems worldwide have been struggling with management of the coronavirus disease 2019
(COVID-19) pandemic. About 58000 publications related to COVID-19\footnote{https://publons.com/publon/covid-19/.}.
 Unfortunately, the disease does not stop spreading and nowadays, about 152 025 280 worldwide infections over 219 countries and territories (May 01, 2021)\footnote{https://www.worldometers.info/coronavirus/}.

In the book \cite{Mart}  basic definitions in epidemiology and some mathematical models such as SI, SIS, SEIS, SARS, SIRS are given.
To study epidemics at the population level, one of the most traditional mathematical model is the Susceptible, Exposed, Infectious, Recovered (SEIR) compartmental model. In \cite{Anca}, it was adapted a traditional SEIR epidemic model to the specific dynamic compartments and epidemic parameters of COVID 19, as it spreads in an age-heterogeneous community.

In classic SEIR models, four compartments are usually considered:
\begin{itemize}
\item[$S(t)$] -   the susceptible population  at
time $t$ (i.e., the class of individuals who are healthy but can contract the disease);
\item[$E(t)$] - the exposed population  (individuals who have contracted the virus but are not yet symptomatic);
\item[$I(t)$] - the infected population  (class of individuals exhibiting signs
and symptoms, have contracted the disease and are now sick with it), it is assumed that infected population are also infectious;
\item[$R(t)$] - the recovered population  (the number of individuals who have recovered and cannot contract the disease again, no longer infect others).
\end{itemize}
In a closed system which does not account for births or deaths, the sum of
these compartments  remains constant in time, i.e.
$$S(t)+E(t)+I(t)+R(t)=N.$$
 Let us consider SEIR model \cite{Anca}:
\begin{equation}
\begin{cases}
\frac{dS}{dt} & =-\beta S(I+qE)/N \\[2mm]
\frac{dE}{dt} & =\beta S(I+qE)/N-E/\delta\\[2mm]
\frac{dI}{dt} & =E/\delta-I/\gamma\\[2mm]
\frac{dR}{dt} & =I/\gamma
\end{cases}\label{eq1}
\end{equation}
where, the parameter $\beta$ is the average number of contacts per person per time, multiplied by the probability of disease
transmission via a contact between a susceptible individual, and an individual carrying the virus. The carrier
can be either infected or exposed, with the fraction ${SI}/{N^2}$ representing the likelihood of an arbitrary contact to
be between a susceptible and an infectious individual, and the fraction ${SE}/{N^2}$ corresponding to the likelihood
of a contact to be between a susceptible and an exposed individual.

The model allows the possibility that a contact with an exposed individual may have different probability of transmission than that made with an infected
individual, which is reflected in the scaling factor $q$). The transition rate at which people are exposed then takes
the form $-d(S/N)/dt=\beta S(I+qE)/N^2$, leading to the first equation of (\ref{eq1}). The rate of
transfer from the exposed to the infectious stage is a fraction $1/\delta$ the number of exposed individuals, where $\delta$
is the average time for an exposed individual to become infectious. The rate of recovery is a fraction $1/\gamma$ the
infectious population, where $\gamma$ is the average time it takes a person to die or recover once in the infectious stage.

This model has already been used in its original form for an early assessment of the epidemic in Wuhan, China \cite{Anca}.

In the equations (\ref{eq1}) we do the following replacements:

$$x=\frac{S}{N}, \ \ y=\frac{E}{N}, \ \ u=\frac{I}{N}, \ \ v=\frac{R}{N}, \ \ a=\frac{1}{\delta}, \ \ b=\frac{1}{\gamma}$$
and obtain

\begin{equation}
\begin{cases}
\frac{dx}{dt} & =-\beta x(u+qy)\\[2mm]
\frac{dy}{dt} & =\beta x(u+qy)-ay\\[2mm]
\frac{du}{dt} & =ay-bu\\[2mm]
\frac{dv}{dt} & =bu
\end{cases}\label{cont}
\end{equation}
where $a>0, b>0, \beta\geq0, q\geq0 $. Here we notice that $x+u+y+v=1$.

In this paper we study discrete time dynamical system, associated to the system
(\ref{cont}), which is generated by the operator $V$ defined as
\begin{equation}\label{disc}
V : \left\{ \begin{alignedat}{1}x^{(1)} & =x-\beta x(u+qy)\\
y^{(1)} & =y-ay+\beta x(u+qy) \\
u^{(1)} & =u-bu+ay\\
v^{(1)} & =v+bu
\end{alignedat}
\right.
\end{equation}

\section{Reduction to Quadratic Stochastic Operators}

\emph{The quadratic stochastic operator} (QSO) \cite{GMR}, \cite{L}, \cite{Rpd} is
a mapping of the simplex
\begin{equation}
S^{m-1}=\{x=(x_{1},...,x_{m})\in\mathbb{R}^{m}:x_{i}\geq0,\sum\limits _{i=1}^{m}x_{i}=1\}\label{2}
\end{equation}
into itself, of the form
\begin{equation}
F:x'_{k}=\sum\limits _{i=1}^{m}\sum\limits _{j=1}^{m}P_{ij,k}x_{i}x_{j},\qquad k=1,...,m,\label{3}
\end{equation}
where the coefficients $P_{ij,k}$ satisfy the following conditions
\begin{equation}
P_{ij,k}\geq0,\quad P_{ij,k}=P_{ji,k},\quad\sum\limits _{k=1}^{m}P_{ij,k}=1,\qquad(i,j,k=1,...,m).\label{4}
\end{equation}

Thus, each QSO $F$ can be uniquely
defined by a cubic matrix $\mathbb{P}=(P_{ij,k})_{i,j,k=1}^{m}$
with conditions (\ref{4}).

Note that each element $x\in S^{m-1}$ is a probability distribution
and each such distribution can be interpreted as a state of the corresponding
biological system.

For a given $\lambda^{(0)}\in S^{m-1}$ the \emph{trajectory}
(orbit) $\{\lambda^{(n)};n\geq0\}$ of $\lambda^{(0)}$ under
the action of QSO (\ref{3}) is defined by
\[
\lambda^{(n+1)}=F(\lambda^{(n)}),\;n=0,1,2,...
\]

The \emph{main problem} in mathematical biology consists in the study of the asymptotical behaviour of the trajectories.

\begin{defn}\label{dr}
A QSO $F$ is called \textbf{regular} if for any initial point $\lambda^{(0)}\in S^{m-1}$,
the limit
\[
\lim_{n\to\infty}F^{n}(\lambda^{(0)})
\]
exists, where $F^n$ denotes $n$-fold composition of $F$ with itself (i.e. $n$ time iterations of $F$).
\end{defn}

\begin{pro}\label{parcon} For the operator (\ref{disc}) we have $V\left(S^{3}\right)\subset S^{3}$ if and only if
\begin{equation}\label{abb}
a, b, \beta \in [0,1], \ \ \mbox{and} \ \ \beta q\leq 1.
\end{equation}
Moreover, under this conditions the operator $V$, defined in (\ref{disc}), is a QSO.
\end{pro}

\begin{proof} Using $x+y+u+v=1$ we rewrite the operator (\ref{disc})
(as in \cite{SS}, \cite{RSH}, \cite{RSHV}):
\begin{equation}
V:\left\{ \begin{alignedat}{1}x^{(1)} & =x(x+y+u+v)-\beta x(u+qy)\\
y^{(1)} & =y(x+y+u+v)(1-a)+\beta x(u+qy) \\
u^{(1)} & =u(x+y+u+v)(1-b)+ay(x+y+u+v)\\
v^{(1)} & =v(x+y+u+v)+bu(x+y+u+v)
\end{alignedat}
\right.
\end{equation}

From this system and definition of QSO we get the following relations:

\begin{equation}
\begin{array}{cccc}
\begin{aligned}{\scriptstyle P_{11,1}} & ={\scriptstyle 1},\\
{\scriptstyle 2P_{14,1}} & ={\scriptstyle 1},\\
{\scriptstyle P_{22,2}} & ={\scriptstyle 1-a},\\
{\scriptstyle 2P_{12,3}} & ={\scriptstyle a},\\
{\scriptstyle 2P_{23,3}} & ={\scriptstyle 1+a-b,} \ \ \\
{\scriptstyle 2P_{34,3}} & ={\scriptstyle 1-b},\\
{\scriptstyle 2P_{23,4}} & ={\scriptstyle b},\\
{\scriptstyle 2P_{34,4}} & ={\scriptstyle 1+b},\\

\end{aligned}

\begin{aligned}{\scriptstyle 2P_{12,1}} & ={\scriptstyle 1-\beta q},\\
 {\scriptstyle 2P_{12,2}} & ={\scriptstyle 1-a+\beta q},\\
 {\scriptstyle 2P_{23,2}} & ={\scriptstyle 1-a},\\
  {\scriptstyle 2P_{13,3}} & ={\scriptstyle 1-b},\\
{\scriptstyle 2P_{24,3}} & ={\scriptstyle a},\\
{\scriptstyle 2P_{13,4}} & ={\scriptstyle b},\\
{\scriptstyle 2P_{24,4}} & ={\scriptstyle 1},\\
{\scriptstyle P_{44,4}} & ={\scriptstyle 1},\\

\end{aligned}

 & \begin{aligned}{\scriptstyle 2P_{13,1}} & ={\scriptstyle 1-\beta},\\
 {\scriptstyle 2P_{13,2}} & ={\scriptstyle \beta},\\
 {\scriptstyle 2P_{24,2}} & ={\scriptstyle 1-a},\\
 {\scriptstyle P_{22,3}} & ={\scriptstyle a},\\
{\scriptstyle P_{33,3}} & ={\scriptstyle 1-b},\\
{\scriptstyle 2P_{14,4}} & ={\scriptstyle 1},\\
{\scriptstyle P_{33,4}} & ={\scriptstyle b},\\
{\scriptstyle \text{other }} & {\scriptstyle P_{ij,k}=0}

\end{aligned}
\end{array}\label{par}
\end{equation}
Now it is easy to see that the conditions (\ref{4}) on $P_{ij,k}$ are equivalent to conditions (\ref{abb}).
\end{proof}

\begin{rk}\label{rem1} Based on the work \cite{Anca}  we note that transmission rate $\beta$ changes between $0.1-0.3,$ infection rate $\frac{1}{\delta}=a$ between $0.07-0.5$ (incubation period 2-14 days),  recovery rate  $\frac{1}{\gamma}=b$ between $0.05-0.1$ (infectious period of 10-20 days). In addition, $\beta q$ is a probability of transmission of the contacting with an exposed individual, so  $\beta q\leq1.$ Thus, the conditions (\ref{abb}) of Proposition \ref{parcon} confirm the biological studies. Therefore, below we consider the operator (\ref{disc}) with condition (\ref{abb}).
\end{rk}

\begin{pro}\label{r} If parameters satisfy (\ref{abb}) then operator (\ref{disc}) is regular.
\end{pro}
\begin{proof} For any initial point $\lambda^{(0)}\in S^3$ the
trajectory $\lambda^{(n)}=(x^{(n)}, y^{(n)}, u^{(n)}, v^{(n)})=V^n(\lambda^{(0)})\in S^3$, is given as
\begin{equation}\label{tr}
 \left\{ \begin{alignedat}{1}
 x^{(n+1)} & =x^{(n)}-\beta x^{(n)}(u^{(n)}+qy^{(n)})\\
y^{(n+1)} & =y^{(n)}-ay^{(n)}+\beta x^{(n)}(u^{(n)}+qy^{(n)}) \\
u^{(n+1)} & =u^{(n)}-bu^{(n)}+ay^{(n)}\\
v^{(n+1)} & =v^{(n)}+bu^{(n)}
\end{alignedat}
\right., \ \ n\geq 0.
\end{equation}
Since $\lambda^{(n)}\in S^3$ we have that all its coordinates
between $0$ and $1$. Therefore, by (\ref{tr}) we have $x^{(n+1)}\leq x^{(n)}$ and $v^{(n+1)}\geq v^{(n)}$.
Thus $x^{(n)}$ and $v^{(n)}$ are convergent as monotone and bounded sequence. To show that $y^{(n)}$ and
$u^{(n)}$ also have limits, consider $a_n=x^{(n)}+y^{(n)}$, $b_n=u^{(n)}+v^{(n)}$ then by (\ref{tr}) we have
$$a_{n+1}=a_n-ay^{(n)}\leq a_n, \ \ b_{n+1}=b_n+ay^{(n)}\geq b_n,$$ i.e., $a_n$ is
decreasing and $b_n$ is increasing. Both sequences are positive and bounded (since $a_n+b_n=1$).
Hence these sequences also have limits. Consequently, from the following equalities it follows that $y^{(n)}$ and
$u^{(n)}$ are convergent
$$y^{(n)}=a_n-x^{(n)}, \ \ u^{(n)}=b_n-v^{(n)}.$$
This completes the proof (according to Definition \ref{dr}).
\end{proof}

\section{Type of fixed points}

Recall that a fixed point of the operator $V$ is a solution of $V(x)=x.$

One can see that the set of fixed points of the operator (\ref{disc}) is
\begin{equation}\label{fix}
{\rm Fix}(V)=\{\Lambda(x)=(x, 0, 0, 1-x): \forall x\in[0, 1]\}.
\end{equation}
 Note that two vertices $e_1=(1, 0, 0, 0), e_2=(0, 0, 0, 1)$ of the simplex $S^3$ belong in ${\rm Fix}(V).$

\begin{rk}\label{ff} Since $V$ is a continuous and regular operator the limit point of each trajectory is a fixed point of the operator.
Therefore, by Proposition \ref{r} we have that independently on the initial point the second and third coordinate (i.e. $y^{(n)}$ and $u^{(n)}$ of
the trajectory has limit
$$\lim_{n\to \infty} y^{(n)}=\lim_{n\to \infty} u^{(n)}=0.$$
\end{rk}

\begin{defn}\label{def1}
\cite{De}. A fixed point $p$ for $F:\mathbb{R}^{m}\rightarrow\mathbb{R}^{m}$
is called \emph{hyperbolic} if the Jacobian matrix $\textbf{J}=\textbf{J}_{F}$
of the map $F$ at the point $p$ has no eigenvalues on the unit
circle.

There are three types of hyperbolic fixed points:
\begin{itemize}
\item[(1)] $p$ is an attracting fixed point if all of the eigenvalues
of $\textbf{J}(p)$ are less than one in absolute value.

\item[(2)] $p$ is an repelling fixed point if all of the eigenvalues
of $\textbf{J}(p)$ are greater than one in absolute value.

\item[(3)] $p$ is a saddle point otherwise.
\end{itemize}
\end{defn}

\begin{pro}\label{fp} The fixed point $\Lambda(x^*)$ is a nonhyperbolic for any $x^*\in[0,1].$
\end{pro}

\begin{proof} In the system (\ref{disc}) we take first three equations:
\begin{equation}
W:\left\{ \begin{alignedat}{1}x^{(1)} & =x-\beta xu-\beta qxy\\
y^{(1)} & =y(1-a)+\beta xu+\beta qxy \\
u^{(1)} & =u(1-b)+ay\\
\end{alignedat}
\right.\label{eq2}
\end{equation}
where $x+y+u\leq1.$\\
For the operator $W$ we calculate a Jacobian as follows:

\[ J=
\left[\begin{array}{cccccc}
1-\beta u-\beta qy & -\beta qx &-\beta x\\
\beta u+\beta qy & 1-a+\beta qx& \beta x \\
0 & a & 1-b \\
\end{array}\right]
\]
and at the fixed point $\Lambda(x^*)$ it has the following form:
\begin{equation}\label{eq3}
 J(\Lambda(x^*))=
\left[\begin{array}{cccccc}
1 & -\beta qx^* &-\beta x^*\\
0 & 1-a+\beta qx^*& \beta x^* \\
0 & a & 1-b \\
\end{array}\right]
\end{equation}

Obviously, for the matrix $J(\Lambda(x^*))$ at least one eigenvalue always equal to 1. Similarly, using $x+y+u+v=1,$ for any triple of equations from the system (\ref{disc}) it can be shown that Jacobian matrix always has an eigenvalue on the unit circle. The proof is completed.
\end{proof}
Let $J(\Lambda(\alpha)),$ be the Jacobian matrix of the operator $W$, at fixed point $\Lambda(\alpha)$, $\alpha\in [0,1]$. Then  (\ref{eq3})  has the following characteristic equation:
   \begin{equation}\label{eq4}
   (1-\mu)[(1-a+ q\alpha\beta-\mu)(1-b-\mu)-a\alpha\beta]=0
\end{equation}
and the roots of this equation are:
\begin{equation}\label{eigv}
\mu_1=1, \ \ \mu_2=1-\frac{a+b-q\alpha\beta-\sqrt{D}}{2}, \ \ \mu_3=1-\frac{a+b-q\alpha\beta+\sqrt{D}}{2},
\end{equation}
where $D=(b-a+q\alpha\beta)^2+4a\alpha\beta\geq0.$

By Proposition \ref{fp}, the fixed point $\Lambda(\alpha)$ is a nonhyperbolic fixed point.

\begin{lem}\label{lem3} Let $\Lambda(\alpha)$ be the fixed point of the operator $W$ and $ \mu_2, \mu_3$ are eigenvalues defined in (\ref{eigv}). Then
$$|\mu_2|=\left\{\begin{array}{lll}
1 \ \ \ \ \  \text{if} \ \ \alpha=\frac{ab}{\beta(a+bq)},\\[2mm]
<1 \ \ \text{if} \ \ \alpha<\frac{ab}{\beta(a+bq)},\\[2mm]
>1 \ \ \text{if} \ \ \alpha>\frac{ab}{\beta(a+bq)},
\end{array}\right.$$  and
$$|\mu_3|<1 \ \ \text{for any} \ \ \alpha\in[0,1].$$

\end{lem}
\begin{proof} \textbf{Step-1}. First we consider the condition $|\mu_2|<1.$
$$|\mu_2|<1 \Leftrightarrow 0<a+b-q\alpha\beta-\sqrt{D}<4$$
since $a+b-q\alpha\beta-4<0$ one has  $a+b-q\alpha\beta-4<\sqrt{D}.$

Let $a+b-q\alpha\beta>0,$ i.e., $\alpha<\frac{a+b}{q\beta}.$ Then
$$\sqrt{D}<a+b-q\alpha\beta \Leftrightarrow 4a\alpha\beta+(b-a+q\alpha\beta)^2<(a+b-q\alpha\beta)^2 \Leftrightarrow \alpha<\frac{ab}{\beta(a+bq)}.$$
Thus, if $\alpha<\min\{\frac{ab}{\beta(a+bq)},\frac{a+b}{q\beta}\}$ then $|\mu_2|<1.$ But always $\frac{ab}{\beta(a+bq)}<\frac{a+b}{q\beta}.$ Assume the opposite:
$$\frac{ab}{\beta(a+bq)}\geq\frac{a+b}{q\beta} \Leftrightarrow (a+b)(a+bq)\leq abq \Leftrightarrow a^2+ab+b^2q\leq0.$$
Hence, $|\mu_2|<1$ if and only if   $\alpha<\frac{ab}{\beta(a+bq)}.$ Moreover, $|\mu_2|=1$ if $\alpha=\frac{ab}{\beta(a+bq)}$ and $|\mu_2|>1$ if $\alpha>\frac{ab}{\beta(a+bq)}.$ \\
Of course, there exists $\alpha$ satisfying the last inequality, because, by Remark  \ref{rem1}, $b\leq\beta$ and from this one has $\frac{ab}{\beta(a+bq)}\leq\frac{ab}{\beta a}=\frac{b}{\beta}\leq1.$\\
\textbf{Step-2}. Now, we consider the condition $|\mu_3|<1.$
$$|\mu_3|<1 \Leftrightarrow 0<a+b-q\alpha\beta+\sqrt{D}<4.$$
If $q\alpha\beta-a-b<0$ i.e., $\alpha<\frac{a+b}{q\beta}$ then left inequality holds true. If $\alpha\geq\frac{a+b}{q\beta}$ then from  $q\alpha\beta-a-b<\sqrt{D}$ one has $\alpha>\frac{ab}{\beta(a+bq)}.$ Thus, the inequality $0<a+b-q\alpha\beta+\sqrt{D}$ is always true. Consequently, we have to check $\sqrt{D}<4-(a+b-q\alpha\beta).$ Note that $4-(a+b-q\alpha\beta)>0.$
$$\sqrt{D}<4-(a+b-q\alpha\beta) \Leftrightarrow 4a\alpha\beta+(b-a+q\alpha\beta)^2<16-8(a+b-q\alpha\beta)+(a+b-q\alpha\beta)^2$$ $$\Leftrightarrow a\alpha\beta+2(a+b-q\alpha\beta)-b(a-q\alpha\beta)-4<0 \Leftrightarrow$$ $$\Leftrightarrow a<\frac{4+2q\alpha\beta-2b-bq\alpha\beta}{2-b+\alpha\beta}=\frac{(2-b)(2+q\alpha\beta)}{2-b+\alpha\beta}<2+q\alpha\beta.$$
But last inequality and therefore  $|\mu_3|<1$ always holds.
This completes the proof of the lemma.
\end{proof}

Let $\bar{x}$ be a fixed point. Then (see \cite{Galor})

$$E^s(\bar{x})=\emph{{\rm span}\{eigenvectors associated with eigenvalues
of modulus$<1$\}}$$ is called a stable eigenspace

$$E^u(\bar{x})=\emph{{\rm span}\{eigenvectors associated with eigenvalues
of modulus$>1$\}}$$ is called an unstable eigenspace

$$E^c(\bar{x})=\emph{{\rm span}\{eigenvectors associated with eigenvalues
of modulus$=1$\}}$$ is called a center eigenspace.

Using Lemma \ref{lem3} and \cite{Galor} we  define the followings  about dimensions of eigenspaces:
$$\text{dim} E^s(\Lambda(\alpha))=1, \text{dim} E^c(\Lambda(\alpha))=2, \text{dim} E^u(\Lambda(\alpha))=0, \ \ \text{if} \ \ \alpha=\frac{ab}{\beta(a+bq)},$$
$$\text{dim} E^s(\Lambda(\alpha))=2, \text{dim} E^c(\Lambda(\alpha))=1, \text{dim} E^u(\Lambda(\alpha))=0, \ \ \text{if} \ \ \alpha<\frac{ab}{\beta(a+bq)},$$
$$\text{dim} E^s(\Lambda(\alpha))=1, \text{dim} E^c(\Lambda(\alpha))=1, \text{dim} E^u(\Lambda(\alpha))=1, \ \ \text{if} \ \ \alpha>\frac{ab}{\beta(a+bq)},$$

Moreover, if $\bar{x}$ is a fixed point, then from the center manifold theory (\cite{George}),  there are the following invariant manifolds:\\
locally stable manifold $W_{loc}^s$ tangent to $E^s$ at $\bar{x};$\\
locally center manifold $W_{loc}^c$ tangent to $E^c$ at $\bar{x};$\\
locally unstable manifold $W_{loc}^u$ tangent to $E^u$ at $\bar{x}.$ \\
locally center-stable manifold $W_{loc}^{cs}$ tangent to $E^c\oplus E^s$ at $\bar{x};$\\
locally center-unstable manifold $W_{loc}^{cu}$ tangent to $E^c\oplus E^u$ at $\bar{x};$\\

The stable and unstable manifolds are unique, but center, center-stable, and center-unstable manifolds may not be unique \cite{Galor}.

\section{The limit points of trajectories}

In this section we study the limit behavior of trajectories of initial point $\lambda^{0}\in S^3$ under
 operator (\ref{disc}), i.e. the sequence $V^n(\lambda^{0})$, $n\geq 1$.
 %Note that since $V$ is a continuous operator, its trajectories have as a limit some fixed point.

\begin{pro} If $\beta=0$ then for any initial point $\lambda^{0}=\left(x^{0},u^{0},y^{0},v^{0}\right)\in S^{3}$
(except fixed point) the trajectory has the following limit
\[
\lim_{n\to\infty}V^{(n)}(\lambda^{0})=(x^0;0;0;1-x^0)
\]
\end{pro}
\begin{proof} If $\beta=0$ then $x^{(n)}=x^0, \ \ y^{(n)}=(1-a)^ny^0\rightarrow0.$ Since $v^{(n+1)}\geq v^{(n)},$ for any $n\in N,$ one has the sequence $v^{(n)}$ has limit as a bounded increasing sequence. Moreover, $u^{(n)}=1-x^{(n)}-y^{(n)}-v^{(n)}$ so the sequence $u^{(n)}$ also has a limit $\bar{u}.$ We get a limit from $u^{(n+1)}=u^{(n)}(1-b)+ay^{(n)}$ and from positiveness of $b$ one has $\bar{u}=0.$ Thus, the proof is completed.
\end{proof}
\begin{rk} For further results, we assume that $\beta>0$ and the conditions of the Proposition \ref{parcon} are satisfied.
\end{rk}
\begin{lem}\label{lem1} The sets
$$X=\{(x, y, u)\in [0,1]^3 : x=0, y+u\leq 1\},$$
$$Y=\{(x, y, u)\in [0,1]^3 : x\leq 1, y=u=0\},$$
$$M=\{(x, y, u)\in [0,1]^3 : x+y+u\leq 1, ay-\beta x(u+qy)> 0, bu-ay> 0\}$$
are invariant sets with respect to operator (\ref{eq2}).
\end{lem}
\begin{proof} The invariantness of $X$ and $Y$ are straightforward.
To show invariantness of $M$, for shortness denote
$$A=ay-\beta x(u+qy), \ \ B=bu-ay.$$

Then first we rewrite $bu^{(1)}-ay^{(1)}$ as
$$bu^{(1)}-ay^{(1)}=b(u-bu+ay)-a(y-ay+\beta x(u+qy))$$
\begin{equation}\label{A}=b(u-B)-a(y-A)=(1-b)B+aA.\end{equation}
%$$=bu-ay-b(bu-ay)+a(ay-\beta x(u+qy))$$
%$$=(bu-ay)(1-b)+a(ay-\beta x(u+qy)).$$

Next rewrite $ay^{(1)}-\beta x^{(1)}(u^{(1)}+qy^{(1)})$ as
$$
ay^{(1)}-\beta x^{(1)}(u^{(1)}+qy^{(1)})=a(y-ay+\beta x(u+qy))$$
$$-\beta(x-\beta x(u+qy))(u-bu+ay+q(y-ay+\beta x(u+qy)))$$
$$=a(y-A)-\beta(x-\beta x(u+qy))(u-B+q(y-A))$$
$$=ay-\beta x(u+qy)-aA+\beta x(B+qA)+\beta^{2} x(u+qy)(u^{(1)}+qy^{(1)})$$
$$=(1-a)A+\beta x(B+qA)+\beta^{2} x(u+qy)(u^{(1)}+qy^{(1)}).$$
Since $u^{(1)}\geq0, y^{(1)}\geq0$ and by the conditions for the parameters (by Proposition \ref{parcon}) in the last expression, each term is non-negative, so this completes the proof of the lemma.

%$$=ay-a(ay-\beta x(u+qy))-\beta(x-\beta x(u+qy))(u+qy-(bu-ay)-q(ay-\beta x(u+qy)))$$
%$$=ay-\beta x(u+qy)-a(ay-\beta x(u+qy))+
%\beta x(bu-ay)+\beta qx(ay-\beta x(u+qy))$$
%$$+\beta^2x(u+qy)^2-\beta^2x(u+qy)(bu-ay)-\beta^2qx(u+qy)(ay-\beta x(u+qy))$$
%$$=(1-a)(ay-\beta x(u+qy))+\beta x(bu-ay)+\beta qx(ay-\beta x(u+qy))$$
%$$+\beta^2x(u+qy)(u-bu+qy)+\beta^2x(u+qy)(ay-qay+\beta qx(u+qy))$$
%$$=(1-a)(ay-\beta x(u+qy))+\beta x(bu-ay)+\beta qx(ay-\beta x(u+qy))$$
%$$+\beta^2xu(u+qy)(1-b)+\beta^2x(u+qy)qy-\beta^2x(u+qy)aqy+\beta^2x(u+qy)(ay+\beta qx(u+qy))$$
%$$
%=(1-a)(ay-\beta x(u+qy))+\beta x(bu-ay)+\beta qx(ay-\beta x(u+qy))+\beta^2xu(u+qy)(1-b)$$
%$$+\beta^2qxy(u+qy)(1-a)+\beta^2x(u+qy)(ay+\beta qx(u+qy)).$$

\end{proof}

\begin{rk} Note that the  invariant set $Y$ belongs to the intersection of two surfaces $ay-\beta x(u+qy)=0, bu-ay=0$ and the extension of $Y$ to a simplex coincides with the set of fixed points $\Lambda(x)=(x;0;0;1-x), x\in [0,1].$
 \end{rk}

 \begin{lem}\label{lem2} For any initial point $(x^0, y^0, u^0)\notin M,$ there exists $k\geq 2$ such that $(x^{(k)}, y^{(k)}, u^{(k)})\in M.$
 \end{lem}

 \begin{proof} Suppose the opposite, for any initial point $(x^0, y^0, u^0)\notin M,$ and for any $n\in N,$  $(x^{(n)}, y^{(n)}, u^{(n)})\notin M.$ It means that
 \begin{equation}\label{s1}
ay^{(n)}-\beta x^{(n)}(u^{(n)}+qy^{(n)})\leq 0\ \ \mbox{or} \ \
bu^{(n)}-ay^{(n)}\leq 0.
\end{equation}
in addition,  by the operator (\ref{disc}) one has

 \begin{equation}\label{s2}
\begin{cases}
y^{(n+1)}=y^{(n)}-[ay^{(n)}-\beta x^{(n)}(u^{(n)}+qy^{(n)})]\\[2mm]
u^{(n+1)}=u^{(n)}-[bu^{(n)}-ay^{(n)}]
\end{cases}
\end{equation}
 from the systems (\ref{s1})  and (\ref{s2})  we have that $y^{(n)}$  or $u^{(n)}$ increasing sequence, but by Remark \ref{ff} both sequences converges to zero. This contradiction completes the proof.
 \end{proof}

%From the system (\ref{disc}) one can see that the sequences $x^{(n)}$ and $v^{(n)}$ have limits as decreasing-bounded and increasing-bounded sequences respectively. Next, it is enough to show that at least one of the sequences $y^{(n)},u^{(n)}$ has a limit.  Here  we consider two possible cases:\\
%\emph{(i)}: for any initial point $\lambda^{0}\in S^{3}$ (except fixed point) there is a number $k\in N$ such that $(x^{(k)};y^{(k)};u^{(k)})\in M.$  By the Lemma \ref{lem1} both bounded sequences $y^{(n)},u^{(n)}$ are decreasing, hence, they have limits.\\
%\emph{(ii)}: for any initial point $\lambda^{0}\in S^{3},$ $(x^{(n)};y^{(n)};u^{(n)})\notin M$ for any  $n\in N.$ It means that at least one of the bounded sequences $y^{(n)},u^{(n)}$ is increasing, i.e., has a limit. We have proved that in any case the operator (\ref{disc}) has a limit. Since, limit point of the operator must be a fixed point, second case does not occur. The proof is finished.

\begin{thm}\label{th1} Assume (\ref{abb}) is satisfied. Then for any initial point $\lambda^{0}=\left(x^{0},u^{0},y^{0},v^{0}\right)\in S^{3}$
(except fixed point) the trajectory has the following limit
\[
\lim_{n\to\infty}V^{(n)}(\lambda^{0})=(\bar{x};0;0;1-\bar{x})
\]
where $\bar{x}$ depends on parameters and initial point $\lambda^0$ and $\bar{x}<\frac{ab}{\beta(a+bq)}.$
\end{thm}

\begin{proof} Existence of the limit is known by Proposition \ref{r}. First, we consider the intersection of two surfaces $ay-\beta x(u+qy)=0, bu-ay=0$ which are represented in the invariant set $M$ (Lemma \ref{lem1}). Solving the system of equations $ay-\beta x(u+qy)=0, bu-ay=0$ we have the parametric equations of two direct lines $x=x, y=0, u=0$ and $x=\frac{ab}{\beta(a+bq)}, y=y, u=ay/b.$ (as in Fig.\ref{fig3}).

 \begin{figure}[h!]
  \centering
  \includegraphics[width=10cm]{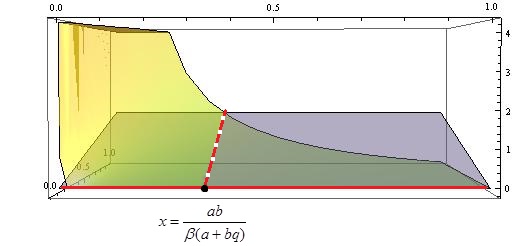}\\
  \caption{Red lines are intersection of surfaces $ay-\beta x(u+qy)=0, bu-ay=0$.}\label{fig3}
\end{figure}

 Based on Lemma \ref{lem1} and Lemma \ref{lem2} we have that $\bar{x}<\frac{ab}{\beta(a+bq)}.$ Biologically it means that total number of susceptible at the limit does not exceed $\frac{ab}{\beta(a+bq)}.$  In Fig. \ref{fig3} dotted red line represents the peaks of disease for initial points not in $M.$\\
Dependence of $\bar{x}$ to initial state and parameters can be represented as following:
from the last equation of the system (\ref{disc}) we find
$$u=\frac{v^{(1)}-v}{b} \Rightarrow u^{(1)}=\frac{v^{(2)}-v^{(1)}}{b},$$
substituting this into third equation we get
 $$y=\frac{v^{(2)}+(b-2)v^{(1)}+(1-b)v}{ab},$$
using them and second equation of the system one has
$$x=\frac{v^{(3)}+(a+b-3)v^{(2)}+(3+ab-2a-2b)v^{(1)}+(a+b-ab-1)v}{\beta(qv^{(2)}+(a+bq-2q)v^{(1)}+(q-bq-a)v)}.$$
Consequently, using $x+y+u+v=1$ we can find the following recurrence formula based on $v^{(n+3)}, v^{(n+2)}, v^{(n+1)}, v^{(n)}$:
$$abv^{(n+3)}=-\beta q\left(v^{(n+2)}\right)^2-\beta(2bq-4q+a+aq)v^{(n+2)}v^{(n+1)}-$$
$$-\beta[q(1-a)(1-b)+(q-a-bq)]v^{(n+2)}v^{(n)} -\beta(a+bq-2q)(a+b-2)\left(v^{(n+1)}\right)^2-$$
$$-\beta[(1-a)(1-b)(a+bq-2q)+(a+b-2)(q-a-bq)]v^{(n+1)}v^{(n)}-$$
$$-\beta(1-a)(1-b)(q-a-bq)\left(v^{(n)}\right)^2-ab[a+b-\beta q-3]v^{(n+2)}-$$
$$-ab[3+ab-2a-2b-\beta(a+bq-2q)]v^{(n+1)}-$$
$$-ab[a+b-ab-1-\beta(q-a-bq)]v^{(n)},$$
where for a given initial point $(x^0,y^0,u^0,v^0)$ using the system (\ref{disc}) one can define $v^{(1)}, v^{(2)}.$
This recursion gives dependence of $v^{(n)}$ and its limit on the initial point. But the problem of finding of an explicate formula
for $\bar x=\bar x(\lambda^{0}) $ remains open.
\end{proof}

\section{Discussion}
In this work, we proved that after some time the COVID-19 virus disappears (without counting mutation). But, if the disease transmission is large enough, then the big part of population will be infectious. Using computer analysis we have studied the disease spreading for concrete values of parameters. For Uzbekistan, we assumed that incubation period is 10 days ($a=0.1$),
infectious period is 15 days ($b=0.066$), transmission rate is $\beta=0.12$, scaling factor is $q=1$ and we get the following results:

(1) peak of the disease is 140 days from the start of the epidemic (real case: 138 days, August 4, 2020 \footnote{https://www.worldometers.info/coronavirus/country/uzbekistan/} ) Fig.\ref{fig1}.

(2) completion of the disease is 300 days from the start of the epidemic (real case: 290 days, January 5, 2021 \footnote{https://www.worldometers.info/coronavirus/country/uzbekistan/} ) Fig. \ref{fig2}.

Note that each coordinates of initial state must be multiplied to total population size $N$ of a considering country.
Finally, the model is working for COVID-19 disease without counting coronavirus mutation, (i.e., new stamps). One aspect of our future work is focused to apply this model and more realized mathematical models for the COVID-19 stamps.

\begin{figure}
  \centering
  \includegraphics[width=6cm]{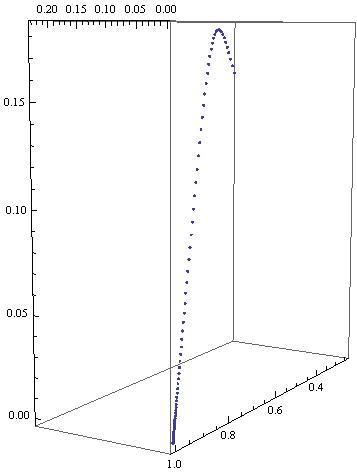}\\
  \caption{$a=0.1, b=0.66, \beta=0.12, q=1, x^{0}=0.99999, u^{0}=0.00001, y^0=0.0, n=140 \, (days)$.}\label{fig1}
\end{figure}

\begin{figure}
  \centering
  \includegraphics[width=6cm]{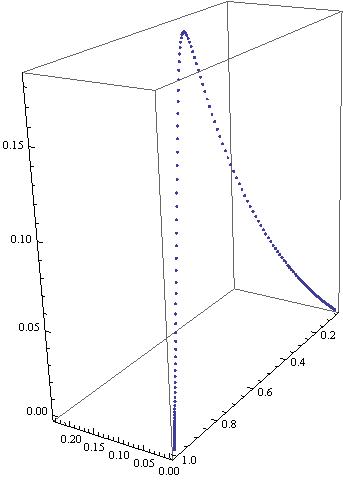}\\
  \caption{$a=0.1, b=0.66, \beta=0.12, q=1, x^{0}=0.99999, u^{0}=0.00001, y^0=0.0, n=300 \, (days)$.}\label{fig2}
\end{figure}

For any initial point (except fixed point) the trajectory located between surfaces $ay-\beta x(u+qy)=0$ and $bu-ay=0$ after the peak of the disease  (Fig.\ref{fig5}).

\begin{figure}
  \centering
  \includegraphics[width=10cm]{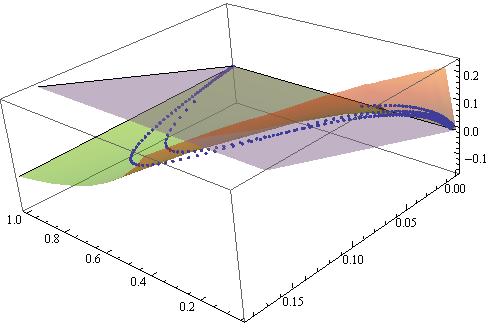}\\
   \caption{$a=0.1, b=0.66, \beta=0.12, q=1$.}\label{fig5}
\end{figure}

In \cite{RSb} we gave a
prediction of COVID-19 pandemic in Uzbekistan. There our method is an observation based projection similar than
the classic Mooreís Law in microelectronics \cite{Lu}.

\end{document}